\documentclass[12pt]{amsart}
\usepackage{newlfont,amsmath,mathrsfs,enumerate,amssymb,array,amscd,xcolor}

\newtheorem{lemma}{Lemma}[section]

\newtheorem{prop}[lemma]{Proposition}
\newtheorem{thm}[lemma]{Theorem}
\newtheorem{cor}[lemma]{Corollary}
\theoremstyle{definition}
\newtheorem{definition}{Definition}[section]
\newtheorem{example}{Example}[section]
\theoremstyle{remark}
\newtheorem{remark}{Remark}

\begin{document}

\title{Action of intertwining operators on pseudo-spherical $K$-types}

\author{Shiang Tang}
\address{Department of Mathematics, University of Utah, Salt Lake City, UT 84112}
\email{tang@math.utah.edu}

\begin{abstract}
In this paper, we give a concrete description of the two-fold cover of a simply connected, split real reductive group and its maximal compact subgroup as Chevalley groups. We define a representation of the maximal compact subgroup called pseudospherical representation, it appears with multiplicity one in the principal series representation. We introduce a family of canonically defined intertwining operators and compute the action of them on pseudospherical $K$-types, obtaining explicit formulas of the Harish-Chandra $c$-function.
\end{abstract}

\maketitle

\section{Introduction}

Assume that $\underline G$ is the split real form of a simply-connected complex algebraic group. It turns out that $\underline G$ admits a unique nontrivial two-fold cover(or double cover) $G$. It is the nonlinear group we wish to study. Such coverings have been studied by many people for a long time. There are several general results about coverings of algebraic groups in [5]. The \emph{pseudospherical principal series representation} is defined for $G$. It is related to a conjectural Shimura correspondence for split real groups, see [1]. Pseudospherical representation can be referred to three definitions: Let $G=PK$ be an Iwasawa decomposition with $P=MAN$ a minimal parabolic subgroup. We have pseudospherical representations of $M$, pseudospherical representations of $K$ and pseudospherical representations of $G$. See definition 3.1. In particular, we are interested in pseudospherical principal series, that is, a principal series representation that contains a pseudospherical $K$-type.

The intertwining operators between two principal series representations has been studied for a long time. An intertwining operator, as an integral operator, reveals many properties of the principal series representation such as reducibility points. The intertwining operator plays an important role in the general Plancherel formula for semisimple Lie groups developed by Harish-Chandra. It is also related to the theory of Eisenstein series. A nice paper on the formalism of this topic can be found in [4]. In this paper, we normalize the intertwining operators between two pseudospherical principal series in a way that it is independent of a choice of $w\in W=N_K(A)/Z_K(A)$ in $N_K(A)$ and obtain a canonical definition. We are interested in the action of intertwining operators on pseudospherical $K$-types. We compute explicitly the \emph{Harish-Chandra $c$-function} associated to this action, which is our main result (Theorem 6.5). There is an analogous result in the $p$-adic case obtained by H.Y.Loke and G.Savin, see [3].

The structure of this paper is arranged as follows: In section 2, we recall some basic facts on Chevalley groups and their covering groups. We define the maximal compact subgroup $K$ of the covering group $G$ using Steinberg symbols. We calculate the structure of $\widetilde{SL}(2,\mathbb R)$, the non-trivial two-fold cover of $SL(2,\mathbb R)$, making a comparison between Kubota cocycles and Steinberg symbols and writing down the exponential map from the Lie algebra to the cover. In section 3, we define the pseudospherical representation following [1] and list some properties regarding the action of $W$ on it. In section 4, we define a family of canonical intertwining operators among pseudospherical principal series. In section 5 we compute the intertwining operators of $\widetilde{SL}(2,\mathbb R)$ which is important for the general groups. Finally, we calculate the action of intertwining operators on pseudospherical $K$-types and obtain our main result in section 6.

ACKNOWLEDGEMENT: This research has been partially supported by an NSF grant DMS 0852429. I am also greatly indebted to my adviser Gordan Savin, without whom none of the works would be possible.

\section{Chevalley groups and their covering groups}

In section 2.1, we recall the well-known construction of the Chevalley groups. In section 2.2, we state a bunch of results we need for the covering group of a Chevalley group. In particular, the generators and relations of the double cover are given in terms of the Hilbert symbol. We define the minimal parabolic subgroup $P=NAM$ and the maximal compact subgroup $K$ in terms of the Steinberg symbol, see proposition 2.6. In section 2.3, we specialize our discussion in 2.2 to the case $\widetilde{SL}(2,\mathbb R)$, we make a comparison between the definition based on Kubota symbols and the definition based on Steinberg symbols, see proposition 2.8. We also compute explicitly an exponential map from the Lie algebra to the cover, see proposition 2.11.

\subsection{Construction of a Chevalley group}

In this section, we recall the construction of Chevalley groups following [5]. Let $\mathfrak g$ be a semisimple Lie algebra over $\mathbb C$, and $\mathfrak h$ a Cartan subalgebra of $\mathfrak g$, and $\Phi$ the corresponding root system. We use $\alpha,\beta,\gamma,...$ to denote the roots. Let $B$ be the Killing form on $\mathfrak g$. Since it is non degenerate, there exists $H_{\alpha}'\in \mathfrak h$ such that $B(H,H_{\alpha}')=\alpha(H)$ for all $H\in \mathfrak h$. Define $(\alpha,\beta)=B(H_{\alpha}',H_{\beta}')$ for all $\alpha,\beta\in\Phi$. The Cartan integer $<\alpha,\beta>$ is defined to be  $2(\alpha,\beta)/(\beta,\beta)$. $\Phi$ is invariant under all reflections $w_{\alpha}$ ($\alpha\in \Phi$) where $w_{\alpha}$ is the reflection across the hyperplane orthogonal to $\alpha$. These reflections generate the Weyl group $W$.

For each $\alpha$, define $H_{\alpha}=\frac{2}{(\alpha,\alpha)}H_{\alpha}'$ and $H_i=H_{\alpha_i}$ where $\Delta=\{\alpha_1,...,\alpha_l\}$ is a set of simple roots. By [5], one can choose $X_{\alpha}\in \mathfrak g_{\alpha}$ such that $[H_{\beta},X_{\alpha}]=<\alpha,\beta>X_{\alpha}$, $[X_{\alpha},X_{-\alpha}]=H_{\alpha}=$ a integer linear combination of the $H_i$'s, and $[X_{\alpha},X_{\beta}]=N_{\alpha\beta}X_{\alpha+\beta}$ where $N_{\alpha\beta}$ is an integer which is $0$ if $\alpha+\beta$ is not a root. The collection of $H_i$'s and $X_{\alpha}$'s is called a \textbf{Chevalley basis} of the complex semisimple Lie algebra $\mathfrak g$. It is important that the integer span $\mathfrak g_{\mathbb Z}$ of the basis elements is stable under the Lie bracket.

Let $L_0$ be the root lattice, i.e, the integer span of all roots in $\Phi$, and $L_1$ the weight lattice, which is the set of all $\mu\in\mathfrak h^*$ such that $\mu(H_{\alpha})\in\mathbb Z$ for all roots $\alpha$. Assume $(\mathfrak g,V)$ is a complex finite dimensional representation of $\mathfrak g$, one can show that its weight lattice $L_V$ is contained between $L_0$ and $L_1$. To construct the Chevalley group based on the representation $(\mathfrak g,V)$, choose a full-rank lattice $M$ in $V$ which is invariant under the set
$$\{X_{\alpha}^n/n!: n\in\mathbb Z_{\geq 0}, \alpha\in\Phi \}$$
where we are thinking of $X_{\alpha}^n/n!$ as a member of $End(V)$. One can show (see [5]) that such a lattice exists. For any field $k$, set $V^k$ to be the vector space $M\otimes_{\mathbb Z}k$ on which $X_{\alpha}^n/n!$ acts in a natural way. Since the representation $V$ has a finite number of weights, there is some $n$ for each $\alpha$ such that $X_{\alpha}^n\in End(V^k)$ is zero. Therefore for $t\in k$ and $\alpha\in \Phi$,
$$x_{\alpha}(t)=exp(tX_{\alpha})=1+tX_{\alpha}+\frac{(tX_{\alpha})^2}{2!}+\frac{(tX_{\alpha})^3}{3!}+...\in GL(V^k)$$ is a finite sum and hence is well-defined.

Define the \textbf{Chevalley group} to be the subgroup $G(k)$ of $GL(V^k)$ generated by $x_{\alpha}(t)$ with $t\in k$, $\alpha\in\Phi$. We say $G$ is \textbf{simply connected} if $L_V=L_1$. Note that this definition is different from simply-connectedness in the topological sense. We assume all Chevalley groups are simply connected for the rest of this paper.

Define $$w_{\alpha}(t)=x_{\alpha}(t)x_{-\alpha}(-t^{-1})x_{\alpha}(t)$$
and $$h_{\alpha}(t)=w_{\alpha}(t)w_{\alpha}(-1)$$ for $t\in k^{\times}$. Define $T$ (the Cartan subgroup, or maximal torus) be the subgroup of $G$ generated by $h_{\alpha}(t)$ with $t\in k^{\times}$, $\alpha\in \Phi$. By [5,lemma 28], $h_{\alpha}(t)$ is multiplicative as a function of $t$, and simply-connectedness implies that any element of $T$ can be written uniquely as $h_1(t_1)h_2(t_2)\cdots h_l(t_l)$ for some $t_1,...,t_l \in k^{\times}$ where $h_i(t_i)=h_{\alpha_i}(t_i)$.

Now let us describe the generators and relations of a simply connected Chevalley group $G$ over $k$:

\begin{align*}
&(A)\:\:\: x_{\alpha}(t)x_{\alpha}(u)=x_{\alpha}(t+u)\\
&(B)\:\:\: (x_{\alpha}(t),x_{\beta}(u))=\prod_{i,j>0,i\alpha+j\beta\in\Phi} x_{i\alpha+j\beta}(c_{ij}t^iu^j)\\
&(B')\:\:\: w_{\alpha}(t)x_{\alpha}(u)w_{\alpha}(-t)=x_{-\alpha}(-t^{-2}u)\\
&(C)\:\:\: h_{\alpha}(t)h_{\alpha}(u)=h_{\alpha}(tu)
\end{align*}
where $c_{ij}$'s are integers depending on $\alpha,\beta$ and the chosen ordering, but not on $t$ or $u$.
By [5,theorem 8], if $\Phi$ is not of type $A_1$, then $A,B,C$ form a complete set of relations for $G$ constructed from $\Phi$ and $k$; If if $\Phi$ is of type $A_1$, then $A,B',C$ form a complete set of relations. By [5,lemma 37], $B'$ is also true when $\Phi$ is not of type $A_1$, and it implies that
\begin{align*}
&w_{\alpha}(t)=w_{-\alpha}(-t^{-1})\\
&w_{\alpha}(1)h_{\alpha}(t)w_{\alpha}(-1)=h_{\alpha}(t^{-1})\\
\end{align*}
which we may use later.

\subsection{Covering groups}

To study the covering group of a simply connected Chevalley group, we need some preparations. First, a central extension of a group $G$ is a couple $(\pi,G')$ where $G'$ is a group, $\pi$ is a homomorphism of $G'$ onto $G$ such that $Ker\: \pi \subset$ center of $G'$. A central extension $(\pi,E)$ of a group $G$ is \emph{universal} if for any central extension $(\pi',E')$ of $G$ there exists a unique homomorphism $\phi: E\to E'$ such that $\pi'\circ \phi=\pi$. It is easy to see that if a universal central extension exists, it is unique up to isomorphism.

\begin{thm}([5,theorem 10])
Let $\Phi$ be an irreducible root system and $k$ a field such that $|k|>4$ and if $rank\:\Phi=1$, then $|k|>9$. Let $G$ be the corresponding simply connected Chevalley group abstractly defined by the relations $A,B,B',C$, and $E$ be the group defined by the relations $A,B,B'$ (we use $B'$ only if $rank\:\Phi=1$), and let $\pi$ be the natural homomorphism from $E$ to $G$, then $(\pi,E)$ is a universal central extension of $G$.
\end{thm}

From now on, we use $x_{\alpha}(t),w_{\alpha}(t),h_{\alpha}(t)$ to denote the elements in the central extension of $G$, and $\underline x_{\alpha}(t),\underline w_{\alpha}(t),\underline h_{\alpha}(t)$ to denote the elements in $G$. They are called \emph{Steinberg symbols} of the Chevalley group.

The next theorem gives a complete description of $C=Ker\:\pi$:

\begin{thm}([5,theorem 12])

Keep the assumptions in the previous theorem. $C=Ker\:\pi$ is isomorphic to the abstract group $A$ generated by the symbols $f(t,u)$ ($t,u\in k^*$) subject to the relations:
\begin{align*}
&(a)\:\:\: f(t,u)f(tu,v)=f(t,uv)f(u,v), f(1,u)=f(u,1)=1\\
&(b)\:\:\: f(t,u)f(t,-u^{-1})=f(t,-1)\\
&(c)\:\:\: f(t,u)=f(u^{-1},t)\\
&(d)\:\:\: f(t,u)=f(t,-tu)\\
&(e)\:\:\: f(t,u)=f(t,(1-t)u)
\end{align*}
and in the case $\Phi$ is not of type $C_n$ ($n\geq 1$) the relations above may be replaced by

\begin{align*}
&(ab')\:\:\: f(t,u)f(t',u)=f(tt',u),f(t,u)f(t,u')=f(t,uu')\\
&(c')\:\:\: f(t,u)=f(u,t)^{-1}\\
&(d')\:\:\: f(t,-t)=1\\
&(e')\:\:\: f(t,1-t)=1.
\end{align*}
The isomorphism is given by $$\phi: f(t,u)\mapsto h_{\alpha}(t)h_{\alpha}(u)h_{\alpha}(tu)^{-1}$$
where $\alpha$ is a fixed long root.
One can write
$$h_{\alpha}(t)h_{\alpha}(u)=f(t,u)h_{\alpha}(tu)$$ if we identify $C=Ker\:\pi$ with $A$ via $\phi$.
\end{thm}

\begin{remark}
Because all long roots are conjugate by $W$, $\phi$ does not depend on the choice of a long root $\alpha$.
\end{remark}

\begin{remark}
These relations are satisfied by the norm residue symbol in class field theory.
\end{remark}

For the application to real groups, we specialize our result to the case when $k=\mathbb R$, and consider the double cover. First, recall the real Hilbert quadratic symbol $(,)_{\mathbb R}$. It is a map from $\mathbb R^*\times \mathbb R^*$ to $\mu_2=\{\pm 1\}$, for $t,u\in \mathbb R^*$, $(t,u)=1$ if and only if $x^2-ty^2-uz^2$ has a nontrivial solution $(x,y,z)\in \mathbb R^3$. It is easy to see that $(t,u)=1$ unless both of $t$ and $u$ are negative. Assume $G'$ is a \textbf{double cover} of $G$, more precisely, a central extension $(p,G')$ of $G$ such that $Ker\:p$ is of order $2$ and such that it \emph{does not split}, i.e, there is no homomorphism $i:G\to G'$ so that $p\circ i=id_G$. Since $(\pi,E)$ is the universal central extension of $G$, there exist a homomorphism $q:E\to G'$ such that $p\circ q=\pi$. $q$ maps $C$ onto $Ker\:p$, that is, $Ker\:p$ is a quotient of $C\cong A$. Passing to quotient, we use $\bar f(t,u)\in \mu_2$ to denote the image of $f(t,u)\in A$. They satisfy $a,b,c,d,e$, so in particular, $G'$ is unique up to isomorphism. On the other hand, the Hilbert symbol $(t,u)$ satisfies the relations that $\bar f(t,u)$ satisfies, hence $\bar f(t,u)=(t,u)$. Thus we have

\begin{cor}
Assume $G$ is a simply connected Chevalley group over $\mathbb R$, then there exists a unique (up to isomorphism) double cover $(p,G')$ of $G$. Moreover, an isomorphism $\phi: \mu_2 \to Ker\:p$ is given by $$(t,u)\mapsto h_{\alpha}(t)h_{\alpha}(u)h_{\alpha}(tu)^{-1}$$
where $\alpha$ is a fixed long root and $(t,u)$ is the real Hilbert quadratic symbol. One can write
$$h_{\alpha}(t)h_{\alpha}(u)=(t,u)h_{\alpha}(tu)$$ by identifying $Ker\:p$ and $\mu_2$ via $\phi$. Combining with $A,B,B'$, we get a complete set of relations for $G'$.
\end{cor}

In the universal cover $E$, let $T$ be the subgroup generated by $h_{\alpha}(t)$, $\alpha\in\Phi, t\in k^{\times}$, it is called the \textbf{metaplectic torus} of $E$. We also refer the image of $T$ in any cover of $G$ as the metaplectic torus. The proposition below lists some relations in $T$.

\begin{prop}
Keep the assumptions in theorem 2.1,2.2. Assume furthermore that $\Phi$ is not of type $C_n$, then we have
$$h_{\alpha}(t)h_{\alpha}(u)=f(t,u)h_{\alpha}(tu)$$ if $\alpha$ is long;
$$h_{\alpha}(t)h_{\alpha}(u)=f(t,u)^{n_{\Phi}}h_{\alpha}(tu)$$ if $\alpha$ is short;
$$(h_{\alpha}(t),h_{\beta}(u))=f(t,u)^{<\alpha,\beta>}$$ if $\alpha,\beta$ are long;
$$(h_{\alpha}(t),h_{\beta}(u))=f(t,u)^{<\alpha,\beta>}$$ if $\alpha$ is long, $\beta$ is short;
$$(h_{\alpha}(t),h_{\beta}(u))=f(t,u)^{<\beta,\alpha>}$$ if $\alpha$ is short, $\beta$ is long;
$$(h_{\alpha}(t),h_{\beta}(u))=f(t,u)^{n_{\Phi}\cdot<\alpha,\beta>}$$ if $\alpha,\beta$ are short. Here $n_{\Phi}=max_{\alpha,\beta\in\Phi}\frac{(\alpha,\alpha)}{(\beta,\beta)}$ and we identify $f(t,u)$ with its image in $C$ via $\phi$.
\end{prop}
\begin{proof}
By [5,lemma 37],
$$(h_{\alpha}(t),h_{\beta}(u))=h_{\beta}(t^{<\beta,\alpha>}u)h_{\beta}(t^{<\beta,\alpha>})^{-1}h_{\beta}(u)^{-1}--(1)$$ for any $\alpha,\beta$. If $\beta$ is long, the right hand side is $f(u,t^{<\beta,\alpha>})^{-1}$, since $\Phi$ is not of type $C_n$, it is equal to $f(t,u)^{<\beta,\alpha>}$. Taking the inverse on both sides, we get $(h_{\beta}(u),h_{\alpha}(t))=f(u,t)^{<\beta,\alpha>}$. Now assume $\beta$ is short, $\alpha$ is long in (1), $h_{\beta}(u)h_{\beta}(t^{<\beta,\alpha>})h_{\beta}(t^{<\beta,\alpha>}u)^{-1}
=(h_{\beta}(u),h_{\alpha}(t))=f(u,t)^{<\alpha,\beta>}=f(u,t^{<\alpha,\beta>})=f(u,t^{<\alpha,\beta>})
=f(u,t^{<\beta,\alpha>})^{\frac{(\alpha,\alpha)}{(\beta,\beta)}}=f(u,t^{<\beta,\alpha>})^{n_{\Phi}}$.
Because $<\beta,\alpha>=\pm 1$, $t^{<\beta,\alpha>}$ runs through all the elements in $k^{\times}$. Finally if both of $\alpha,\beta$ are short,
$(h_{\alpha}(t),h_{\beta}(u))=(h_{\beta}(u)h_{\beta}(t^{<\beta,\alpha>})h_{\beta}(t^{<\beta,\alpha>}u)^{-1})^{-1}
=f(u,t^{<\beta,\alpha>})^{-n_{\Phi}}=f(t,u)^{n_{\Phi}\cdot<\beta,\alpha>}$.
\end{proof}

\begin{remark}
In particular, assume $G$ is a real group and $G'$ is its double cover. The relations above are still true if we replace $f(t,u)$ by $(t,u)$. Because the Hilbert symbol $(t,u)$ is bimultiplicative, by the proof of prop 2.4, one can remove the assumption that $\Phi$ is not of type $C_n$. \end{remark}

\begin{prop}
\begin{align*}
&h_{\alpha}(t)x_{\beta}(u)h_{\alpha}(t)^{-1}=x_{\beta}(t^{<\beta,\alpha>}u)\\
&w_{\alpha}(1)h_{\beta}(t)w_{\alpha}(-1)=h_{w_{\alpha}\beta}(ct)h_{w_{\alpha}\beta}(c)^{-1}\\
\end{align*}
where $c=c(\alpha,\beta)=\pm 1$ is independent of $t$ and $u$.
\end{prop}
\begin{proof}
This is just part of [5,lemma 37].
\end{proof}

The next proposition gives a description of maximal compact subgroup in the setting of Chevalley groups:
\begin{prop}
Assume $k=\mathbb C$ or $\mathbb R$. Then there exists an automorphism $\sigma$ of $E$ so that $\sigma x_{\alpha}(t)=x_{-\alpha}(-t)$ for any $\alpha\in\Phi$, and an automorphism $\underline\sigma$ of $G$ so that $\underline\sigma\:\underline x_{\alpha}(t)=\underline x_{-\alpha}(-t)$ for any $\alpha\in\Phi$. We have $\sigma h_{\alpha}(t)=h_{\alpha}(t^{-1})$, $\underline \sigma \:\underline h_{\alpha}(t)=\underline h_{\alpha}(t^{-1})$.
Moreover, The group $K$ of fixed points of $\sigma$ is a subgroup of $E$ containing $C=Ker \:\pi$ and the group $\underline K$ of fixed points of $\underline \sigma$ is a maximal compact subgroup of $G$. $K=\pi^{-1}(\underline K)$.
\end{prop}

\begin{proof}
This is basically [5,theorem 16], which proves the existence of $\underline \sigma$ and $\underline K$ for $G$. In particular, $x_{\alpha}(t)\mapsto x_{-\alpha}(-t)$, $\forall \alpha\in \Phi$ preserves the relations (A) and (B). Hence $\underline \sigma$ can be lifted to an automorphism of $E$ which we denote by $\sigma$ such that $\sigma x_{\alpha}(t)=x_{-\alpha}(-t)$ for any $\alpha\in\Phi$. By the definition of $w_{\alpha}(t)$, $\sigma w_{\alpha}(t)=w_{-\alpha}(-t)$. So $\sigma h_{\alpha}(t)=\sigma w_{\alpha}(t)w_{\alpha}(-1)=\sigma w_{\alpha}(t)\sigma w_{\alpha}(-1)= w_{-\alpha}(-t)w_{-\alpha}(1)$. Since $w_{\alpha}(t)=w_{-\alpha}(-t^{-1})$ for any $\alpha\in \Phi$, $t\in k^{\times}$, the last term is $w_{\alpha}(t^{-1})w_{\alpha}(-1)=h_{\alpha}(t^{-1})$. Thus $\sigma h_{\alpha}(t)=h_{\alpha}(t^{-1})$ as in the linear case. $\sigma$ fixes $C$. In fact, with the notation of theorem 2.2, $C$ is generated by $f(t,u)$, $t,u\in k^{\times}$ if we identify the groups $A$, $C$ via $\phi$, and $h_{\alpha}(t)h_{\alpha}(u)=f(t,u)h_{\alpha}(tu)$. Let $\sigma$ act on both sides, one has $h_{\alpha}(t^{-1})h_{\alpha}(u^{-1})=\sigma f(t,u)h_{\alpha}(t^{-1}u^{-1})$, which implies that $\sigma f(t,u)=f(t^{-1},u^{-1})$. By relation (c) in theorem 2.2, $f(t^{-1},u^{-1})=f(u,t^{-1})=f(t,u)$ and hence $\sigma$ fixes $C$.
\end{proof}

For the rest of this paper, we use $\underline G$ to denote a simply connected Chevalley group over $\mathbb R$, $G$ the \emph{double cover} of $\underline G$. For any subgroup $H$ of $G$, let $\underline H$ be the image of $H$ under the covering projection $p:G\to \underline G$. Define the real metaplectic torus $T$ to be the subgroup of $G$ generated by $h_{\alpha}(t)$ with $\alpha\in\Phi$ and $t\in \mathbb R^*$. Let $A\cong(\mathbb R^{\times})^l$ be the subgroup of $T$ generated by $h_{\alpha}(t)$ with $\alpha\in\Phi$, $t>0$, here $l$ is the rank of $\Phi$. By remark 3, $p|_A:A\to \underline A$ is an isomorphism, and hence for simplicity, we just use $A$ to denote this group. Let $M$ be the subgroup of $T$ generated by $h_{\alpha}(-1)$ with $\alpha\in\Phi$. It is easy to see that $A$ is in the center of $T$, and $T$ is the direct product of $A$ and $M$. $M$ is a central extension of $\underline M \cong (\mathbb Z/2\mathbb Z)^{l}$ by $\mu_2=\{\pm 1\}$. Let $\Delta$ be a set of simple roots, $\Phi^+$ be the corresponding set of positive roots. Let $N$ be the group generated by $x_{\alpha}(t)$ with $\alpha\in\Phi^{+},t\in\mathbb R$. Then $p|_N: N\to \underline N$ is an isomorphism, and hence for simplicity, we just use $N$ to denote this group. Define $P$ to be subgroup of $G$ generated by $N$ and $T$, which we call a minimal parabolic subgroup(or Borel subgroup). We have the Langlands decomposition $P=NAM$. By proposition 2.6, there exists an automorphism $\sigma$ of $G$ so that $\sigma x_{\alpha}(t)=x_{-\alpha}(-t)$, $\forall \alpha \in \Phi$. Similarly for $\underline G$. The group $K$ of fixed points of $\sigma$ is a maximal compact subgroup of $G$ which is the double cover of $\underline K$. It is easy to see that $M$(resp. $\underline M$) is a subgroup of $K$(resp. $\underline K$). One has $Z_{\underline K}(A)=\underline M$ which implies that $Z_{K}(A)=M$. Define the Weyl group $W$ to be $N_K(A)/Z_K(A)=N_K(A)/M$. $W$ is isomorphic to $N_{\underline K}(A)/\underline M$.

\begin{lemma}
$w_{\alpha}(1)\in N_K(A)$ for any $\alpha\in\Phi$, and their images in $N_K(A)/M$ generate $W$.
\end{lemma}
\begin{proof}
Since $\sigma w_{\alpha}(1)=w_{-\alpha}(-1)=w_{\alpha}(1)$, $w_{\alpha}(1)\in K$. Also by the second relation in prop 2.5, it normalizes $A$. $w_{\alpha}(1)$ corresponds to the reflection $s_{\alpha}$ through the hyperplane determined by $\alpha$, which gives an isomorphism between $W=N_K(A)/Z_K(A)$ the Weyl group $\hat W$ defined in the abstract root system setting. In particular, $w_{\alpha}(1)$, $\alpha\in \Phi$ generate $W$.
\end{proof}

\subsection{The group $SL(2,\mathbb R)$ and its double cover $\widetilde{SL}(2,\mathbb R)$}

In this section, we recall some basic facts about $SL(2,\mathbb R)$ and its double cover $\widetilde{SL}(2,\mathbb R)$, which are important for the study of representation theory of general covering groups.

$\underline G=SL(2,\mathbb{R})$ may be described in Steinberg symbols: let $X=\begin{pmatrix}0&1\\0&0\end{pmatrix}$, $Y=\begin{pmatrix}0&0\\1&0\end{pmatrix}$, $H=\begin{pmatrix}1&0\\0&-1\end{pmatrix}$ be the $\mathfrak{sl}_2$ triple. For $t\in\mathbb R$, define
\begin{align*}
&\underline x(t)=exp(tX)=\begin{pmatrix}1&t\\0&1\end{pmatrix}\\
&\underline y(t)=exp(tY)=\begin{pmatrix}1&0\\t&1\end{pmatrix}\\
&\underline w(t)=\underline x(t)\underline y(-t^{-1})\underline x(t)=\begin{pmatrix}0&t\\-t^{-1}&0\end{pmatrix}\\
&\underline h(t)=\underline w(t)\underline w(-1)=\begin{pmatrix}t&0\\0&t^{-1}\end{pmatrix}\\
\end{align*}
Let $N$ be the subgroup generated by $\underline x(t)$, $t\in\mathbb R$, $A$ be the subgroup generated by $\underline h(t)$, $t>0$. $\underline K=SO(2)$ consists of $r_{\phi}=\begin{pmatrix} cos\phi&-sin\phi\\sin\phi&cos\phi\\ \end{pmatrix}$, $\phi\in\mathbb R$. $\underline G=NA\underline K$.
Let $\underline M=\{\underline h(\pm 1)\}\in \underline K$, then the subgroup $\underline P$ of upper triangular matrices has the Langlands decomposition $\underline P=NA\underline M$.

By Corollary 2.3, there exists a unique nontrivial double cover $G=\widetilde{SL}(2,\mathbb R)$ of $\underline G=SL(2,\mathbb{R})$, that is, a central extension of $\underline G$ by $\mu_2=\{\pm 1\}$. We use $p$ to denote the covering map. It is generated by the symbols $x(t)$, $y(t)$ satisfying the same relations as that of $\underline G$ except that $h(t)h(u)=(t,u)h(tu)$ where $(,)$ is the real Hilbert quadratic symbol. $\phi: N\to \widetilde{SL}(2,\mathbb R)$, $\underline x(t)\mapsto x(t)$, $t\in \mathbb R$, is a group homomorphism; $\psi: A\to \widetilde{SL}(2,\mathbb R)$ $\underline h(t)\mapsto h(t)$, $t>0$, is also a group homomorphism. Moreover, $\phi$ is the only homomorphism from $N$ to $\widetilde{SL}(2,\mathbb R)$ satisfying $p\circ \phi=Id_N$. In fact, assume $\phi'$ is another one, consider $f: N\to \mu_2$, $n\mapsto \phi(n)\phi'(n)^{-1}$, we have $f(\underline x(t))=f(\underline x(t/2)^2)=f(\underline x(t/2))^2=1$. So $f$ is trivial, whence $\phi=\phi'$. Similar fact holds for $\psi$. We still denote the images of $\phi, \psi$ by $N,A$. Let $K$ be the subgroup fixed by the automorphism $\sigma$ of $G$ where $\sigma$ sends $x(t)$ to $y(-t)$, $y(t)$ to $x(-t)$. $K$ is a double cover of $SO(2)$. We have the Iwasawa decomposition $G=NAK$. Let $M=\{\pm h(\pm 1)\}\subset K$, it is isomorphic to $C_4$, the cyclic group of order four, and it is an extension of $\underline M$ by $\mu_2$. $P=NAM$ is an extension of $\underline P$ by $\mu_2$.

We may also describe the group structure of $\widetilde{SL}(2,\mathbb R)$ using Kubota cocycle. The only reason why we introduce this is that Kubota cocycle makes some calculations involving $K$ more explicit. It will be used in section 5. $\widetilde{SL}(2,\mathbb{R})=SL(2,\mathbb{R})\times \mu_2$ as a set, the group law is given by \begin{center} $(g,\varepsilon)(g',\varepsilon')=(gg',\varepsilon\varepsilon'c(g,g'))$ \\ \end{center} $c$ is called the \emph{Kubota cocycle}, which is given by the formula \begin{center} $c(g,g')=(x(g),x(g'))(-x(g)x(g'),x(gg'))$\\ \end{center} where

$$ x(\begin{pmatrix} a&b\\c&d\\ \end{pmatrix})=
\begin{cases}
c & \text{if $c\neq0$}\\
d & \text{if $c=0$}
\end{cases}
$$
and $(,)$ is the quadratic Hilbert symbol. $x(t)\mapsto (\underline x(t),1)$, $y(t)\mapsto (\underline y(t),1)$ gives an isomorphism between the two definitions. Direct calculation using Kubota cocycle shows that $w(t)\mapsto (\underline w(t),1)$ and $h(t)\mapsto (\underline h(t),sgn(t))$. Thus we may write
\begin{prop}
\begin{align*}
&x(t)=(\underline x(t),1)\\
&y(t)=(\underline y(t),1)\\
&w(t)=(\underline w(t),1)\\
&h(t)=(\underline h(t),sgn(t))\\
\end{align*}

\end{prop}

The exponential map $$\underline{exp}: \mathfrak{sl}(2,\mathbb R)\to SL(2,\mathbb R)$$ is given by the exponents of matrices. In particular,
\begin{align*}
&\underline{exp}(tX)=\underline x(t)\\
&\underline{exp}(tH)=\underline h(e^{t})\\
&\underline{exp}(-tZ)=r_t\\
\end{align*}
where $Z=X-Y$.

\begin{prop}
Let $\underline e:\mathbb R\to SO(2)$ be the homomorphism sending $\phi$ to $r_{\phi}$. Then there exists a unique homomorphism $e:\mathbb R \to K$ so that $p\circ e=\underline e$. It is given by $e(\phi)=(r_{\phi},\epsilon(\phi/2))$ where $\epsilon: \mathbb{R}/2\pi\mathbb{Z}\to \pm 1$ is defined by $\epsilon(\theta)=sgn(sin\theta sin2\theta)$ when $\theta\neq 0,\pi/2,\pi,3\pi/2$, $\epsilon(0)=1$, $\epsilon(\pi/2)=-1$, $\epsilon(\pi)=-1$, $\epsilon(3\pi/2)=1$.
\end{prop}
\begin{proof}
It is clear that $e$ is of the form appear in the statement for some $\epsilon: \mathbb{R}/2\pi\mathbb{Z}\to \pm 1$. By working out $exp(\theta)exp(\theta)=exp(2\theta)$, one sees that $\epsilon(\theta)=(x(r_{\theta}),x(r_{\theta}))(-1,x(r_{2\theta}))=(-1,x(r_{\theta})x(r_{2\theta}))=sgn(x(r_{\theta})x(r_{2\theta}))$. Direct calculations show that $x(r_{\theta})x(r_{2\theta})=sin\theta sin2\theta$ when $\theta\neq 0,\pi/2,\pi,3\pi/2$, $0\leq \theta <2\pi$ and $\epsilon(0)=1$, $\epsilon(\pi/2)=-1$, $\epsilon(\pi)=-1$, $\epsilon(3\pi/2)=1$.
\end{proof}

\begin{cor}
For any integer $n$, $\sigma_{n/2}: K\to S^1$, $(r_{\phi},\epsilon(\phi/2))\mapsto e^{in\phi/2}$ is a character of $K$. In particular $\sigma=\sigma_{n/2}|_M$ is a character of $M$ such that $\sigma(I,1)=1,\sigma(-I,-1)=i^n,\sigma(I,-1)=(-1)^n,\sigma(-I,1)=(-i)^n$.
\end{cor}

There exists a unique exponential map $$exp: \mathfrak{sl}(2,\mathbb R)\to \widetilde{SL}(2,\mathbb R)$$ such that $p\circ exp=\underline{exp}$. In fact, for any $X\in \mathfrak{sl}(2,\mathbb R)$, let $\gamma(t)$ be the unique one parameter subgroup of $\underline G$ whose tangent vector at the identity is equal to $X$. Since $p$ is a covering map, $t\mapsto \gamma(t)$ can be partially lifted to a continuous map $\tilde\gamma: I\to G$ such that it pushes forward to $\gamma|_I$ for some neighborhood $I\subset \mathbb R$ around $0$ and $\tilde\gamma(0)=1\in G$. Since $\gamma$ is a continuous homomorphism, one can extend $\tilde\gamma$ to a homomorphism from $\mathbb R$ to $G$ which lifts $\gamma$. We define $exp(X)$ to be $\tilde\gamma(X)$. We have

\begin{prop}

\begin{align*}
&exp(tX)=(\underline x(t),1)\\
&exp(tH)=(\underline h(e^{t}),1)\\
&exp(-tZ)=(r_t,\epsilon(t/2))\\
\end{align*}

\end{prop}
\begin{proof}
The first two are obvious and the last equality follows from proposition 2.9.
\end{proof}

\subsection{Connections between $\widetilde{SL}(2,\mathbb R)$ and general covering groups}

Let $G$ be the unique nontrivial two-fold cover of a split real group $\underline G$. For each root $\alpha$,
$$\Phi_{\alpha}:\widetilde{SL}(2,\mathbb R)\to G$$
is defined to be the homomorphism sending $x(t)$ to $x_{\alpha}(t)$, $y(t)$ to $x_{-\alpha}(t)$, $h(t)$ to $h_{\alpha}(t)$.

We now make the following definition from [1], which will be used later:
\begin{definition}
A root $\alpha$ is said to be \emph{metaplectic} if $\Phi_{\alpha}$ does not factor to $SL(2,\mathbb R)$.
\end{definition}

The following proposition follows directly from the first two equations in proposition 2.4:
\begin{prop}
If $G$ is not of type $G_2$, then $\alpha$ is metaplectic iff it is long. If $G$ is of type $G_2$, then all roots are metaplectic.
\end{prop}

\section{Pseudospherical Representations}
For each $\alpha\in\Phi$, let $m_{\alpha}=h_{\alpha}(-1)\in G$ and $Z_{\alpha}=X_{\alpha}-X_{-\alpha}\in\mathfrak g$.
We have $exp(-\pi Z_{\alpha})=m_{\alpha}$ by proposition 2.8 and 2.11.
The following definition is from definition 4.9 and lemma 4.11 of [1]:
\begin{definition}(\textbf{Pseudospherical representations})
An irreducible representation $\sigma$ of $M$ is pseudospherical if the eigenvalues of $\sigma(m_{\alpha})$ belong to $\{\pm i\}$ when $\alpha$ is a metaplectic root, and $\{1\}$ otherwise;
An irreducible representation $\mu$ of $K$ is pseudospherical if the eigenvalues of $d\mu(iZ_{\alpha})$ belong to $\{\pm 1/2\}$ when $\alpha$ is a metaplectic root, and $\{0\}$ otherwise;
A representation of $G$ is pseudospherical if it contains a pseudospherical $K$-type.
\end{definition}

\begin{remark}
When $G=\widetilde{SL}(n,\mathbb R)$, the double cover of $SL(n,\mathbb R)$, the \emph{Spinor representation}  representation of $K=Spin(n)$ is pseudospherical.
\end{remark}

\begin{remark}
If $G$ is simply laced or of type $G_2$, then every irreducible genuine representation of $M$ is pseudospherical. In fact, all roots are metaplectic in this case, and so $m_{\alpha}^2=h_{\alpha}(-1)h_{\alpha}(-1)=-1\in\mu_2\subset Z(G)$. So $\sigma(m_{\alpha})^2=\sigma(-1)=-I$ and hence its eigenvalues are $\pm i$ with multiplicities.
\end{remark}

Also notice that eigenvalues in the pseudospherical conditions for $M$ and $K$ are compatible: $log(1)=2\pi i\mathbb Z$, $log(\pm i)=\pm i \frac{\pi}{2}+2\pi i\mathbb Z$.

\begin{example}
In the case $\widetilde{SL}(2,\mathbb R)$, the representation $\mu=\sigma_{1/2}$ is a pseudospherical representation of $K=\widetilde{SO}(2)$, whose restriction $\sigma=\sigma_{1/2}|_M$ to $M$ is a pseudospherical representation of $M$. In fact,
$$d\mu(Z_{\alpha})=lim_{t\to 0} \frac{\mu(exp(tZ_{\alpha}))-1}{t}$$ But $\mu(exp(tZ_{\alpha}))=\sigma_{1/2}(r_{-t},\epsilon(-t/2))=e^{-it/2}$, the limit is $lim_{t\to 0} \frac{e^{-it/2}-1}{t}=-i/2$. So $d\mu(iZ_{\alpha})=1/2$; $\sigma(m_{\alpha})=\sigma(h_{\alpha}(-1))=\sigma(r_{\pi},\epsilon(\pi/2))=e^{i\pi/2}=i$.
\end{example}

Below is a fundamental fact on pseudospherical representations:
\begin{thm}([1,proposition 5.2])
Let $\sigma$ be a pseudospherical representation of $M$. There is a unique pseudospherical representation $\mu_{\sigma}$ of $K$ such that $\mu_{\sigma}|_M=\sigma$ and this defines a bijection between pseudospherical representation of $M$ and $K$.
\end{thm}

Now we want to define an action of $W$ on irreducible representations $(\sigma,V)$ of $M$ that do not factor through $\underline M$ (or equivalently, $\sigma(-1)\neq 1$), we call such representations \emph{genuine representations}. We use $\Pi_g(M)$ to denote the set of isomorphism classes of genuine representations of $M$. We will show that $W$ fixes every isomorphism class of irreducible genuine pseudospherical representation of $M$.

\begin{prop}([1])
Let $Z(M)\supset \mu_2$ be the center of $M$, $\Pi_g(Z(M))$ be the set of genuine characters of $Z(M)$, that is, those characters $\chi$ with the property $\chi(-1)=-1$. For every $\chi\in \Pi_g(Z(M))$, there is a unique representation $\sigma(\chi)$ of $M$ such that $\sigma(\chi)|_{Z(M)}=\chi\cdot I$. The map $\chi\mapsto \sigma(\chi)$ defines a bijection $\Pi_g(Z(M))\to\Pi_g(M)$. The dimension of $\sigma(\chi)$ is $|M/Z(M)|^{1/2}$ and $Ind_{Z(M)}^M(\chi)\cong |M/Z(M)|^{1/2}\sigma(\chi)$.
\end{prop}
\begin{proof}
This is proved in [1], but we repeat the argument for completeness.

The key point is that if $\sigma$ is a genuine representation of $M$, then the character $tr\:\sigma$ is supported on $Z(M)$. In fact, suppose $m$ does not belong to $Z(M)$. Choose $h\in M$ so that $hmh^{-1}\neq m$. Since $\underline M$ is abelian, $p(hmh^{-1})=p(h)p(m)p(h)^{-1}=p(m)$, so $hmh^{-1}=-m$. Taking the trace on both sides we have $tr\:\sigma(m)=\chi(-1)tr\:\sigma (m)$. Since $\chi$ is genuine, $\chi(-1)=-1$, so $tr\:\sigma(m)=0$.

Therefor every irreducible genuine representation of $M$ is uniquely determined by its central character. Fix $\chi\in\Pi_g(Z(M))$. Let $I(\chi)=Ind_{Z(M)}^M(\chi)$. This has central character $\chi$, so it is a multiple of the irreducible representation $\sigma(\chi)$ of $M$ with central character $\chi$. Put $I(\chi)=n\sigma(\chi)$. By Frobenius reciprocity,
$$Hom_M(I(\chi),I(\chi))=Hom_{Z(M)}(I(\chi)|_{Z(M)},\chi)$$ which has dimension $|M/Z(M)|$. On the other hand, by Schur's lemma,
$$Hom_M(I(\chi),I(\chi))=Hom_M(n\sigma(\chi),n\sigma(\chi))$$ which has dimension $n^2$. Therefore $n=|M/Z(M)|^{1/2}$ and the dimension of $\sigma(\chi)$ is $|M/Z(M)|^{1/2}$.
\end{proof}

Since $N_K(A)$ acts on $M$ by conjugation, it also acts on its center $Z(M)$, which factors down to $W=N_K(A)/M$. Thus we have an action of $W$ on $\Pi_g(Z(M))$. By the proposition above, this gives rise to an action of $W$ on $\Pi_g(M)$. More precisely, pick $\hat w\in N_K(A)$ which is a representative of $w\in W$, then $\hat w\sigma(m)=\sigma(\hat w^{-1}m \hat w)$ is a representation of $M$. Up to isomorphism, it is independent of the choice of a representative of $w$ because different representatives $\hat w$ give the same central character, hence isomorphic representations. Therefore one can denote this representation by $w\sigma$, as an isomorphism class in $\Pi_g(M)$.

\begin{prop}
The action of the Weyl group $W$ on the isomorphism classes of irreducible genuine representations of $M$ fixes each isomorphism class of pseudospherical representations.
\end{prop}
\begin{proof}
Assume $(\sigma,V)$ is a genuine representation of $M$. $\forall w\in W$, choose a representative $\hat w$ of $w$ in $N_K(A)$. By theorem 3.1, there is a unique pseudospherical representation $(\mu_{\sigma},V)$ of $K$ such that $\mu_{\sigma}|_M=\sigma$. Define $\phi:V\to V$, $v\mapsto \mu_{\sigma}(\hat w^{-1})v$, we have
$$\phi(\mu_{\sigma}(k)v)=\mu_{\sigma}(\hat w^{-1})\mu_{\sigma}(k)v=\mu_{\sigma}(\hat w^{-1}k\hat w)\mu_{\sigma}(\hat w^{-1})v=(\hat w\mu_{\sigma})(k)\phi(v)$$ for any $k\in K$. Thus $\phi$ is a $K$-isomorphism, restricting it to $M$, we get $\sigma\cong (\hat w\mu_{\sigma})|_M=\hat w\sigma$.
\end{proof}

\section{Principal series representations and Intertwining operators}
In this section, let $G$ be the double cover of a split real group. We define the principal series representation of $G$, and the intertwining operator. Most of the results are well-known in the linear group case, see [4]. The discussion for covering groups is almost identical to the linear case. The highlight is, the intertwining maps can be defined in a canonical way using the theory of pseudospherical representations.

Let $\chi$ be a character of $A$, and $\delta$ be the modular character of $A$ such that
$$\int_{N}f(a^{-1}na)dn=\delta(a)\int_{N}f(n)dn$$ for any $a\in A$ and any compact supported function $f$ on $N$. Here we fix a Haar measure on $N$, which is topologically isomorphic to $\mathbb R^{|\Phi^+|}$. $\delta$ depends on $N$, we will write $\delta_N$ instead of $\delta$ when needed. $\delta$ is equal to the product of the roots in $\Phi^+$, considered as multiplicative characters of $A$. Let $(\sigma, V)$ be a pseudospherical representation of $M$.
\begin{definition}
Define $I(P,\sigma, \chi)$, the space of principal series, to be the space of smooth functions $f:G\to V$ such that
$$f(namx)=\delta(a)^{1/2}\chi(a)\sigma(m)f(x)$$  $\forall n\in N$, $\forall a\in A$, $\forall m\in M$, $\forall x\in G$. $G$ acts on $I(P,\sigma, \chi)$ by right translation: $\rho(g)f(x)=f(xg)$. This defines a representation of $G$ called the \emph{principal series representation, or induced representation} of $G$. For simplicity, we denote this representation by $I(\sigma,\chi)$, $I(\chi)$ when there is no confusion.
\end{definition}

Assume $\chi$ is a character of $A$, $\forall w\in N_K(A)$, $w\chi(a)=\chi(w^{-1}aw)$ is another character of $A$, this action factors down to $W$. Note that $w_1(w_2\chi)=(w_1w_2)\chi$. In other words, we have an action of the Weyl group $W$ on $\Pi(A)=$ the set of characters of $A$.

By theorem 3.1, there is an irreducible representation $(\mu_{\sigma},V)$ of $K$ such that $\mu_{\sigma}|_M=\sigma$. For any $f\in I(P,\sigma,\chi)$ and any $w\in W$, pick a representative $\hat w\in N_K(A)$ of $w$, define a function $M(w,\sigma,\chi)f$ in the following way:
$$M(w,\sigma,\chi)f(x)=\mu_{\sigma}(\hat w)\int_{N\cap \hat w N\hat w^{-1}\backslash N}f(\hat w^{-1}nx)dn$$
Note that $n\to f(\hat w^{-1}nx)$ is left $N\cap \hat w N\hat w^{-1}$-invariant, so the integral makes sense. Also it is well-defined, i.e, independent of the choice of a representative of $w$ in $N_K(A)$ due to the normalizing factor $\mu_{\sigma}$. For simplicity, we write $w$ in place of $\hat w$ when there is no other confusion. Let us remark that $N_w=N\cap wNw^{-1}\backslash N$ correspond to those positive roots who are sent to negative by $w^{-1}$, and it has one to one correspondence with $B\backslash Bw^{-1}N$.

Let $S(w)$ be the set of $\chi$ so that the above integral is absolutely convergent for any $x\in G$, $f\in I(P,\sigma,\chi)$.
We are going to show that $M(w,\sigma,\chi)$ maps $I(P,\sigma,\chi)$ into $I(P,\sigma,w\chi)$ for $\chi\in S(w)$. This map is called \emph{intertwining map}. For simplicity, we denote this map by $M(w,\chi)$ or $M(w)$ sometimes.

\begin{lemma}
Let $w$ be an element in $W$, $\delta_w$ be a character of $A$ such that $$\int_{N_w}f(a^{-1}na)dn=\delta_w(a)\int_{N_w}f(n)dn$$  for any $a\in A$ and any integrable function $f$ on $N$. Then $(w\delta)^{1/2}\delta_w=\delta^{1/2}$.
\end{lemma}
\begin{proof}
For a simple reflection $w$, take $Q=P\cup Pw^{-1}P$ and $L,U$ be its Levi factor and unipotent radical. Note that $U=N\cap wNw^{-1}$. We have $\delta_{N}=\delta_{U}\delta_{N/U}$ and $\delta_{wNw^{-1}}=\delta_{U}\delta_{wNw^{-1}/U}$. But $\delta_{wNw^{-1}}=w\delta_N$ and $\delta_{wNw^{-1}/U}=\delta_{N/U}^{-1}$, the conclusion now follows from simple algebraic manipulations.
\end{proof}

\begin{prop}Assume $\chi\in S(w)$, then $M(w,\sigma,\chi)$ maps $I(P,\sigma,\chi)$ into $I(P,\sigma,w\chi)$.
\end{prop}
\begin{proof}
$M(w,\sigma,\chi)f(nx)=M(w,\sigma,\chi)f(x)$ is obvious.
\begin{align*}
M(w,\sigma,\chi)f(ax)
&=\mu_{\sigma}(w)\int_{N_w} f(w^{-1}nax)dn\\
&=\mu_{\sigma}(w)\int_{N_w} f((w^{-1}aw)w^{-1}(a^{-1}na)x)dn\\
&=\mu_{\sigma}(w)w\delta(a)^{1/2}w\chi(a)\int_{N_w} f(w^{-1}(a^{-1}na)x)dn\\
&=\mu_{\sigma}(w)w\delta(a)^{1/2}w\chi(a)\delta_w(a)\int_{N_w} f(w^{-1}nx)dn\\
&=\delta(a)^{1/2}w\chi(a)M(w,\sigma,\chi)f(x)\\
\end{align*}

\begin{align*}
M(w,\sigma,\chi)f(mx)
&=\mu_{\sigma}(w)\int_{N_w} f(w^{-1}nmx)dn\\
&=\mu_{\sigma}(w)\int_{N_w} f((w^{-1}mw)w^{-1}(m^{-1}nm)x)dn\\
&=\mu_{\sigma}(w)\sigma(w^{-1}mw)\int_{N_w} f(w^{-1}(m^{-1}nm)x)dn\\
&=\sigma(m)\mu_{\sigma}(w)\int_{N_w} f(w^{-1}nx)dn\\
&=\sigma(m)M(w,\sigma,\chi)f(x)\\
\end{align*}

\end{proof}

Assume that the Haar measures on $N_w$'s are normalized so that when $l(w_1w_2)=l(w_1)+l(w_2)$,
$$\int_{N_{w_1w_2}}f(n)dn=\int_{N_{w_1}\times N_{w_2}}f(w_1n_2w_1^{-1}n_1)dn_1dn_2$$ for any integrable function $f$ on $N_{w_1w_2}$. Under this assumption, the following proposition holds:

\begin{prop}
Assume $w_1,w_2\in W$ such that $l(w_1w_2)=l(w_1)+l(w_2)$, then
$$S(w_1w_2)=S(w_2)\cap w_2^{-1}S(w_1)$$ and for $\chi\in S(w)$ which is regular(only fixed by the trivial element in $W$),
$$M(w_1,\sigma,w_2\chi)\circ M(w_2,\sigma,\chi)=M(w_1w_2,\sigma,\chi)$$
\end{prop}
\begin{proof}
Since $\chi$ is regular, by Frobenius reciprocity, the dimension of $Hom_G(I(\chi),I(w\chi))$ is one for any $w\in W$. So it suffices to show that
$(M(w_1)\circ M(w_2)f)(1)=M(w_1w_2)f(1)$.
\begin{align*}
(M(w_1)\circ M(w_2)f)(1)
&=\mu_{\sigma}(w_1)\int_{N_{w_1}}(M(w_2)f)(w_1^{-1}n_1)dn_1\\
&=\mu_{\sigma}(w_1)\mu_{\sigma}(w_2)\int_{N_{w_1}}dn_1\int_{N_{w_2}}f(w_2^{-1}n_2w_1^{-1}n_1)dn_2\\
&=\mu_{\sigma}(w_1w_2)\int_{N_{w_1}}dn_1\int_{N_{w_2}}f(w_2^{-1}w_1^{-1}\cdot w_1n_2w_1^{-1}n_1)dn_2\\
\end{align*}
By the assumption on the Haar measures, the last expression is equal to
$$\mu_{\sigma}(w_1w_2)\int_{N_{w_1w_2}}f(w_2^{-1}w_1^{-1}n)dn=M(w_1w_2)f(1)$$
\end{proof}

\section{Representations of $\widetilde{SL}(2,\mathbb R)$}
We carry out the detailed study of principal series and intertwining maps in the $SL_2$ case first, which is the fundamental building block of the general case. The results in 5.1 are well-known but we list them here for the purpose of making a comparison with the non-linear case.

\subsection{Linear case}
The results and calculations in this section are well known, we present them for the purpose of completeness and exposition.

Let $\underline G=SL(2,\mathbb R)$, $\underline P$ be the standard parabolic subgroup with Langlands decomposition $\underline P=NA\underline M$. $\underline M$ only two characters, let $\sigma$ be any of them. For any complex number $s$, define a character $\chi$ of $A$:
$\chi(a)=a^s$ where $a=diag(a,a^{-1})$. The modular character $\delta(a)$ of $A$ is $a^2$. So the space $I(\underline P,\sigma,\chi)$ of principal series in this case is the collection of functions $f$ such that
$$f(namx)=a^{s+1}\sigma(m)f(x)$$ For simplicity, we denote it by $I(\sigma,s)$.
Let $\underline K=SO(2)$. For any $n\in\mathbb Z$, $\tau_n(r_{\phi})=e^{in\phi}$ is a character of $\underline K$. Define $f_s^n$ to be such that
$$f_s^n(nak)=a^{s+1}\tau_n(k)$$
Then $f_s^n\in I(\tau_n|_{\underline M},s)$. When $n$ is even, $\tau_n|_{\underline M}$ is trivial, denoted by $\sigma_0$; When $n$ is odd, $\tau_n|_{\underline M}$ is nontrivial, denoted by $\sigma_1$. We say $f_s^n$ is of $\underline K$-type $n$.

The Weyl group $W$ is of order two, let $w$ be its nontrivial element. Now we define the intertwining map $M(\sigma,s):I(\sigma,s)\to I(\sigma,-s)$:
When $\sigma=\sigma_0$,
$$M(\sigma,s)f(x)=\int_{N}f(wnx)dn$$
When $\sigma=\sigma_1$,
$$M(\sigma,s)f(x)=\tau_1(w)^{-1}\int_{N}f(wnx)dn$$ for $f\in I(\sigma,s)$. It does not depend on the choice of a representative element of $w$ in $N_K(A)$. $M(\sigma,s)f_s^n=c_n(s)f_{-s}^n$ for some constant $c_n(s)$. We have $c_n(s)=(M(\sigma,s)f_s^n)(1)$. The following proposition is well-known, we will give a proof of a more general proposition in the next subsection.

\begin{prop}
$$c_n(s)=\sqrt{\pi}\frac{\Gamma(\frac{s}{2})\Gamma(\frac{s+1}{2})}{\Gamma(\frac{s+n+1}{2})\Gamma(\frac{s-n+1}{2})}$$
\end{prop}

\subsection{Nonlinear case}
Let $G=\widetilde{SL}(2,\mathbb R)$, $P=NAM$ be its standard parabolic subgroup which is the double cover of $\underline P$. Let $K$ be the double cover of $\underline K=SO(2)$. We are going to study the principal series of $G$ and calculate the intertwining map using the Kubota cocycle.
Let $\sigma$ be a character of $M$, define a character $\chi$ of $A$:
$$\chi(a)=a^s$$ where $s\in \mathbb C$, $a=(diag(a,a^{-1}),1)\in A$. Since $\delta_N(a)=a^2$ and $\delta_{\bar N}(a)=a^{-2}$, $I(P,\sigma,\chi)$, which we denote by $I(\sigma,s)$ for simplicity, consists of functions $f$ so that
$$f(namx)=a^{s+1}\sigma(m)f(x)$$ $I(\bar P,\sigma,\chi)$, which we denote by $\bar I(\sigma,s)$ for simplicity, consists of functions $f$ so that
$$f(\bar namx)=a^{s-1}\sigma(m)f(x)$$

For any $n\in\mathbb Z$, define $f_s^{n/2}$ to be such that
$$f_s^{n/2}(nak)=a^{s+1}\sigma_{n/2}(k)$$ where $\sigma_{n/2}$ is a character of $K$ defined in corollary 2.9.
Then $f_s^{n/2}\in I(\sigma_{n/2}|_M,s)$. We say $f_s^{n/2}$ is of $K$-type $n/2$. Similarly, define $\bar f_s^{n/2}$ to be such that
$$\bar f_s^{n/2}(\bar nak)=a^{s-1}\sigma_{n/2}(k)$$
Then $\bar f_s^{n/2}\in \bar I(\sigma_{n/2}|_M,s)$.

There are four different characters $\sigma$ of $M$: $\sigma_{0}|_M,\sigma_{1/2}|_M,\sigma_{1}|_M,\sigma_{3/2}|_M$. Define the intertwining map
$M(\sigma,s):I(\sigma,s)\to I(\sigma,-s)$
$$M(\sigma,s)f(x)=\sigma_{i/2}(w)^{-1}\int_{N}f(wnx)dn$$ if $\sigma=\sigma_{i/2}|_M$, $i=0,1,2,3$. This definition is canonical.
Define $T:\bar I(\sigma,s)\to I(\sigma,-s)$ by
$$Tf(x)=\sigma_{i/2}(w)^{-1}f(wx)$$ if $\sigma=\sigma_{i/2}|_M$, $i=0,1,2,3$.
Also define intertwining maps $A(\sigma,s):I(\sigma,s)\to \bar I(\sigma,s)$ such that
$$A(\sigma,s)f(x)=\int_{\bar N}f(\bar nx)d\bar n$$
and $\bar A(\sigma,s):\bar I(\sigma,s)\to I(\sigma,s)$ such that
$$\bar A(\sigma,s)f(x)=\int_{N}f(nx)dn$$
Then we have
$$M(\sigma,s)=T\circ A(\sigma,s)$$
$M(\sigma,s)$ sends $f_s^{n/2}$ to $c_{n/2}(s)f_{-s}^{n/2}$ for some constant $c_{n/2}(s)$. It is called \emph{Harish-Chandra c-function} sometimes.
It is easy to see that $c_{n/2}(s)=A(\sigma,s)f_s^{n/2}(1)$. For simplicity, we use $I(s)$ in place of $I(\sigma,s)$ and $A(s)$ in place of $A(\sigma,s)$ sometimes.

Define a pairing $(,): I(s)\times I(-\bar s)\to \mathbb C$:
$$(f,g)=\int_K f(k)\overline{g(k)} dk$$
There is also a pairing $(,): \bar I(s)\times \bar I(-\bar s)\to \mathbb C$ defined using the same formula.
The following lemma follows from formal calculations:
\begin{lemma}
For $f\in I(s)$ and $g\in \bar I(-\bar s)$, we have
$(A(s)f,g)=(f,\bar A(-\bar s)g)$.
\end{lemma}

\begin{prop}
For those $s\in i\mathbb R$ such that $I(s)$ is irreducible, $\bar A(s)\circ A(s)$ is a non-negative constant.
\end{prop}
\begin{proof}
$\bar A(s)\circ A(s)$ is a constant by Schur's lemma, say $\lambda(s)$. When $s\in i\mathbb R$, $s=-\bar s$, so in lemma 5.2, $(A(s)f,g)=(f,\bar A(s)g)$. Taking $g=A(s)f$, we get $(A(s)f,A(s)f)=(f,\bar A(s)\circ A(s)f)=\bar \lambda(s)(f,f)$, hence $\lambda(s)$ is non-negative.
\end{proof}

Below is a nice proposition of the $c$-function:
\begin{prop}
$c_{n/2}(s)=c_{-n/2}(s)$.
\end{prop}

\begin{proof}
Let $d=diag(1,-1)\in GL_2$, the conjugation action of $d$ on $SL_2$ satisfies $dr_{\phi}d^{-1}=r_{-\phi}$. This action lifts to the covering group and it gives an inverse on $K$. Hence the conjugation of functions by $d$ gives an intertwining map $I(\sigma,s)\to I(\sigma^{-1},s)$ which we denote by $d(s)$. We have
$$M(\sigma^{-1},s)\circ d(s)=d(-s)\circ M(\sigma,s)$$
In fact, for any $f\in I(\sigma,s)$,
$$M(\sigma^{-1},s)\circ d(s)f(x)=\sigma_{-n/2}(w)^{-1}\int_N(d(s)f)(wnx)dn$$
$$=\sigma_{-n/2}(w)^{-1}\int_Nf(dwnxd^{-1})dn=\sigma_{n/2}(w)\int_Nf(w^{-1}dnxd^{-1})dn$$
$w^{-1}=mw$ for some $m\in M$, so the last expression is
$$\sigma_{n/2}(w)\sigma(m)\int_Nf(wdnxd^{-1})dn=\sigma_{n/2}(w^{-1})\int_Nf(wdnxd^{-1})dn$$

On the other hand,
$$d(-s)\circ M(\sigma,s)f(x)=\sigma_{n/2}(w)^{-1}\int_Nf(wndxd^{-1})dn$$
$$=\sigma_{n/2}(w)^{-1}\int_Nf(wdnxd^{-1})dn$$
hence the two operators are equal.

Now take $f=f^{n/2}_s$, $M(\sigma^{-1},s)\circ d(s)f_s^{n/2}=c_{-n/2}(s)f_{-s}^{-n/2}$, and $d(-s)\circ M(\sigma,s)f_s^{n/2}=c_{n/2}(s)f_{-s}^{-n/2}$, thus $c_{n/2}(s)=c_{-n/2}(s)$.
\end{proof}

Now we calculate $c_{n/2}(s)$:

\begin{prop}
$$c_{n/2}(s)=\sqrt{\pi}\frac{\Gamma(\frac{s}{2})\Gamma(\frac{s+1}{2})}{\Gamma(\frac{s+1}{2}+\frac{n}{4})\Gamma(\frac{s+1}{2}-\frac{n}{4})}$$
\end{prop}

\begin{proof}
$c_{n/2}(s)=\int_{\bar N}f_s^{n/2}(\bar n)d\bar n$.
For $\bar n=\begin{bmatrix} 1&0\\t&1\\ \end{bmatrix}\in \bar N$, $$\bar n=\begin{bmatrix} 1&x\\0&1\\ \end{bmatrix}\begin{bmatrix} y^{1/2}&0\\0&y^{-1/2}\\ \end{bmatrix}(r_{\phi},1)=\begin{bmatrix} y^{1/2}cos\phi+xy^{-1/2}sin\phi&-y^{1/2}sin\phi+xy^{-1/2}cos\phi\\y^{-1/2}sin\phi&y^{-1/2}cos\phi\\ \end{bmatrix}$$.

$$\bar n\cdot i=\frac{i}{ti+1}=\frac{t}{t^2+1}+\frac{1}{t^2+1}i=x+yi$$ so $a(t)=y^{1/2}=\frac{1}{\sqrt{t^2+1}}$, $tan\phi=t$, $\phi=\phi(t)=arctant$. Then,
$$c_{n/2}(s)=\int_{\mathbb R}f_s^{n/2}(\begin{bmatrix} 1&0\\t&1\\ \end{bmatrix})dt=\int_{\mathbb R}\frac{1}{(t^2+1)^{(s+1)/2}}\sigma_{n/2}(r_{\phi(t)},1)dt$$
$\epsilon(\phi(t)/2)=sgn(sin(\phi(t)/2)sin(\phi(t)))=1$ since $\phi(t)\in (-\pi/2,\pi/2)$, hence $\sigma_{n/2}(r_{\phi(t)},1)=e^{in\phi(t)/2}$. Finally, by substituting those expressions into the last integral, we get $$c_{n/2}(s)=\int_{\mathbb{R}}\frac{1}{(t^2+1)^{(s+1)/2}}e^{in(arctant)/2}=\int_{\mathbb{R}}\frac{1}{(t^2+1)^{(s+1)/2}}(\frac{1-it}{\sqrt{t^2+1}})^{-n/2}dt$$
Now the proposition follows from the lemma below.
\end{proof}

\begin{lemma}
For any $n\in\mathbb Z$, $$\int_{\mathbb{R}}\frac{1}{(t^2+1)^{(s+1)/2}}(\frac{1-it}{\sqrt{t^2+1}})^{-n/2}dt=\sqrt{\pi}\frac{\Gamma(\frac{s}{2})\Gamma(\frac{s+1}{2})}{\Gamma(\frac{s+1}{2}+\frac{n}{4})\Gamma(\frac{s+1}{2}-\frac{n}{4})}$$
\end{lemma}

\begin{proof}
The integral is absolutely convergent for $Re(s)>0$, the integrand is equal to $$(1+it)^{(-2s-n-2)/4}(1-it)^{(-2s+n-2)/4}$$ which we denote by $f(t)$. By Lebesgue's dominated convergence theorem, $$lim_{y\to 0}\int_{\mathbb{R}}f(t)e^{-ity}dt=\int_{\mathbb{R}}f(t)dt$$
Let $2u=(2s+n+2)/4$, $2v=(2s-n+2)/4$, By [2],
$$\hat{f}(y)=\int_{\mathbb{R}}f(t)e^{-ity}dt=2\pi 2^{-u-v}\Gamma(2v)^{-1}y^{u+v-1}W_{v-u,1/2-v-u}(2y)$$ for $y>0$. Here the Whittaker function
$$W_{\rho,\sigma}(z)=\frac{\Gamma(-2\sigma)}{\Gamma(1/2-\sigma-\rho)}M_{\rho,\sigma}(z)+\frac{\Gamma(2\sigma)}{\Gamma(1/2+\sigma-\rho)}M_{\rho,-\sigma}(z)$$ where
\begin{align*}
&M_{\rho,\sigma}(z)=z^{1/2+\sigma}e^{-(1/2)z}F(1/2+\sigma-\rho,2\sigma+1,z)\\
&F(a,b,z)=1+\sum_{k\geq 1}\frac{a(a+1)...(a+k-1)}{b(b+1)...(b+k-1)}\frac{z^k}{k!}
\end{align*}
So $$W_{v-u,1/2-v-u}(2y)=\frac{\Gamma(-1+2u+2v)}{\Gamma(2u)}(2y)^{1-u-v}e^{-y}F(1-2v,2-2u-2v,2y)$$
$$+\frac{\Gamma(1-2u-2v)}{\Gamma(1-2v)}(2y)^{u+v}e^{-y}F(2u,2u+2v,2y)$$
Thus $$y^{u+v-1}W_{v-u,1/2-v-u}(2y)\to \frac{\Gamma(-1+2u+2v)}{\Gamma(2u)}2^{1-u-v}$$ as $y\to 0$.
It follows that
\begin{align*}
\int_{\mathbb{R}}f(t)dt=lim_{y\to 0}\hat{f}(y)
&=2\pi 2^{-u-v}\Gamma(2v)^{-1}\frac{\Gamma(-1+2u+2v)}{\Gamma(2u)}2^{1-u-v}\\
&=\pi 2^{1-s}\frac{\Gamma(s)}{\Gamma(\frac{s+1}{2}+\frac{n}{4})\Gamma(\frac{s+1}{2}-\frac{n}{4})}
\end{align*}
By the double formula, $$\Gamma(s)=(1/\sqrt{\pi})2^{s-1}\Gamma(\frac{s}{2})\Gamma(\frac{s+1}{2})$$ hence $$\int_{\mathbb{R}}f(t)dt=\sqrt{\pi}\frac{\Gamma(\frac{s}{2})\Gamma(\frac{s+1}{2})}{\Gamma(\frac{s+1}{2}+\frac{n}{4})\Gamma(\frac{s+1}{2}-\frac{n}{4})}$$ \end{proof}

Now we consider a slightly more general situation, which will be used in next section. Let $(\sigma, V)$ be a finite dimensional representation of $M$, which is the restriction of a representation $(\mu,V)$ of $K$, $I(\sigma,s)$ be the space of functions $f:G\to V$ such that
$f(namg)=a^{s+1}\sigma(m)f(g)$.
For $f\in I(\sigma,s)$, define $$M(s)f(x)=\mu(w)^{-1}\int_N f(wnx)dn$$
By proposition 3.5, $M(s)$ maps $I(\sigma,s)$ into $I(\sigma,-s)$. For $v\in V$, define $f_s^v(nak)=a^{s+1}\mu(k)v$. Then $v\mapsto f_s^v$ is an embedding of $(\mu,V)$ into $I(\sigma,s)$, as a $K$-subrepresentation.

\begin{prop}
Assume $(\mu,V)$ is a direct sum of $\sigma_{\pm n/2}$ for a fixed integer $n$. Then $(M(s)f_s^v)(1)=c_{n/2}(s)v$.
\end{prop}
\begin{proof}
If $v$ belongs to one of those summands, then by the definition of $M(s)$ and proposition 5.5, $(M(s)f_s^v)(1)=c_{n/2}(s)v$. Because $c_{n/2}(s)=c_{-n/2}(s)$, this is valid for any $v\in V$.
\end{proof}

\section{Action of intertwining operators on pseudo-spherical $K$-types}
This section contains the main result of this paper. Let $G$ be the unique nontrivial two-fold cover of a split real group $\underline G$. Assume $\sigma$ is a pseudospherical representation of $M$ and $\mu_{\sigma}$ is the pseudospherical representation of $K$ corresponding to $\sigma$.  We note that the multiplicity of $\mu_{\sigma}$ in  $I(P,\sigma,\chi)$ is one and then calculate the action of the intertwining operator on it, in particular, the Harish-Chandra $c$-function.

\begin{lemma}
As $K$-representation, the multiplicity of $\mu_{\sigma}$ in $I(P,\sigma,\chi)$ is 1.
\end{lemma}
\begin{proof}
It is easy to see that as a $K$-representation, $I(P,\sigma,\chi)$ is isomorphic to $Ind_M^K(\sigma)$. By Frobenius reciprocity, $Hom_K(\mu_{\sigma},Ind_M^K(\sigma))=Hom_M(\sigma,\sigma)$, which is isomorphic to $\mathbb C$ by Schur's lemma.
\end{proof}

Let $\phi$ be the unique element in $Hom_K(\mu_{\sigma},I(P,\sigma,\chi))$ such that $(\phi v)(1)=v$, $\forall v\in V$ and $\psi$ be the unique element in $Hom_K(\mu_{\sigma},I(P,\sigma,w\chi))$ such that $(\psi v)(1)=v$, $\forall v\in V$. Then $M(w,\sigma,\chi)(\phi v)=c\cdot (\psi v)$ for some nonzero constant $c\in \mathbb C$ which does not depend on $v$.

Let $s=(s_1,...,s_l)\in \mathbb C^l$ and take $\chi=\chi_s$ to be the character of $A$ such that $$\chi_s(h_1(t_1)...h_l(t_l))=t_1^{s_1}...t_l^{s_l}$$ $t_i>0$. We write $I(P,\sigma,s)$ instead of $I(P,\sigma,\chi)$. Let $ws\in\mathbb C^l$ be such that $w\chi_s=\chi_{ws}$. We write $M(w,s)$ for the intertwining map instead of $M(w,\sigma,\chi_s)$.

\begin{lemma}
Define a function $f_{P,\mu_{\sigma},s}^v:G\to V$ such that
$$f_{P,\mu_{\sigma},s}^v(nak)=\chi_s(a)\delta_N(a)^{1/2}\mu_{\sigma}(k)v$$
Then $f_{P,\mu_{\sigma},s}^v$ is well-defined and $f_{P,\mu_{\sigma},s}^v\in I(P,\sigma,s)$. For simplicity, we write $f_s^v$ instead when there is no confusion.
\end{lemma}
\begin{proof}
Since the Iwasawa decomposition is unique(It is not true in the $p$-adic case), $f_s^v$ is well-defined.
$f_s^v(nx)=f_s^v(x)$ is evident. For any $a\in A$, $f_s^v(ax)=f_s^v(an(x)a(x)k(x))$. Since $T$ normalizes $N$, it is equal to $\chi_s(a)\delta_N(a)^{1/2}f_s^v(x)$. Finally, since $T$ normalizes $N$ and $A$ is contained in the central of $T$, $f_s^v(mx)=f_s^v(mn(x)a(x)k(x))=f_s^v(n'(x)ma(x)k(x))=f_s^v(n'(x)a(x)mk(x))=\sigma(m)f_s^v(x)$. Thus $f_s^v\in I(P,\sigma,s)$.
\end{proof}

\begin{lemma}
Define a map $\phi: \mu_{\sigma}\to I(P,\sigma,s)$, $v\mapsto f_s^v$, then $\phi$ is a $K$-intertwining map.
\end{lemma}

\begin{proof}
We need to show $\phi(\sigma(k)v)=R(k)\phi(v)$. Assume for $x\in G$, $x=n(x)a(x)k(x)$ is the Iwasawa decomposition of $x$.
$$\phi(\sigma(k)v)(x)=f_s^{\sigma(k)v}(x)=\chi_s(a(x))\delta_N(a(x))^{1/2}\sigma(k(x))\sigma(k)v$$
On the other hand,
\begin{align*}
R(k)\phi(v)(x)
&=f_s^v(xk)\\
&=\chi_s(a(x))\delta_N(a(x))^{1/2}\sigma(k(x)k)v\\
&=\chi_s(a(x))\delta_N(a(x))^{1/2}\sigma(k(x))\sigma(k)v
\end{align*}
proved the identity.
\end{proof}

\begin{prop}
Assume $\sigma$ is a genuine pseudospherical representation of $M$, then $\mu_{\sigma}|_{K_{\alpha}}=m\sigma_{1/2}\oplus m'\sigma_{-1/2}$ for some integers $m$, $m'$ when $\alpha$ is metaplectic, and $\mu_{\sigma}|_{K_{\alpha}}=m\cdot 1$ for some integer $m$ when $\alpha$ is not metaplectic. Here $K_{\alpha}=\Phi_{\alpha}(\widetilde{SO}(2))$.
\end{prop}
\begin{proof}
For each $\alpha$, $K_{\alpha}$ is generated by $exp(tZ_{\alpha})$, $t\in\mathbb R$. By definition 3.1, the eigenvalues of $\mu(exp(tZ_{\alpha}))$ are $e^{\pm it/2}$ with multiplicities for $\alpha$ metaplectic and $1$ otherwise. On the other hand, for each $n\in \mathbb Z$, $\sigma_{n/2}:K_{\alpha}\to S^1$, $exp(tZ_{\alpha})\mapsto e^{-int/2}$ is a character of $K_{\alpha}$ and they exhaust all the characters of $K_{\alpha}\cong S^1$. Thus $\mu_{\sigma}|_{K_{\alpha}}$ is a direct sum of $\sigma_{\pm 1/2}$ when $\alpha$ is metaplectic and $1$ otherwise.
\end{proof}

Let $G_{\alpha}=\Phi_{\alpha}(\widetilde{SL}(2,\mathbb R))\subset G$, then $G_{\alpha}\cong\widetilde{SL}(2,\mathbb R)$ when $\alpha$ is metaplectic, and $G_{\alpha}\cong SL(2,\mathbb R)$ when $\alpha$ is not metaplectic. Let $T_{\alpha}$ be the image of the metaplectic torus of $\widetilde{SL}(2,\mathbb R)$, $N_{\alpha}$ be the image of the unipotent radical of the standard parabolic subgroup of $\widetilde{SL}(2,\mathbb R)$. Consider $Q=P\cup Pw_{\alpha}P$ where $P=NT=NAM$ is a minimal parabolic subgroup of $G$, then $U=N\cap w_{\alpha}Nw_{\alpha}^{-1}$ is the unipotent radical of $Q$. We have $\delta_N(t)=\delta_U(t)\delta_{N/U}(t)$ for $t\in T$. In particular, take $t\in T_{\alpha}$, we get $\delta_U(t)=1$ hence $\delta_N(t)=\delta_{N/U}(t)=\delta_{N_{\alpha}}(t)$. Thus $\delta_N(t)=\delta_{N_{\alpha}}(t)$ for $t\in T_{\alpha}$.

Now we get to the main result of this paper, a similar result on double covers of $p$-adic groups can be found in [3].
\begin{thm}(\textbf{Action of intertwining operators on pseudospherical $K$-types})
Let $M(w,s): I(P,\sigma,s)\to I(P,\sigma,ws)$ be the intertwining map. Then
$M(w,s)f_{s}^v=c(w,s)f_{ws}^v$ for some constant $c(w,s)$.
Moreover, let $\Delta=\{\alpha_1,...,\alpha_l\}$ be the set of simple roots, $w_i=w_{\alpha_i}$, then in the case when $\Phi$ is simply laced or $G_2$,
$$c(w_{i},s)=c_{1/2}(s_i)$$ for all $i$;
Otherwise,
$$c(w_{i},s)=c_0(s_i)$$ when $\alpha_i$ is short and
$$c(w_{i},s)=c_{1/2}(s_i)$$ when $\alpha_i$ is long.
Here for $\nu\in\mathbb C$,
$$c_0(\nu):=\sqrt{\pi}\frac{\Gamma(\frac{\nu}{2})}{\Gamma(\frac{\nu+1}{2})}$$
$$c_{1/2}(\nu):=\sqrt{\pi}\frac{\Gamma(\frac{\nu}{2})\Gamma(\frac{\nu+1}{2})}{\Gamma(\frac{\nu}{2}+\frac{3}{4})\Gamma(\frac{\nu}{2}+\frac{1}{4})}$$
\end{thm}

\begin{proof}
The idea is reduction to the $SL_2$ case.

The multiplicities of $(\mu_{\sigma},V)$ in $I(P,\sigma,s)$ and $I(P,\sigma,ws)$ are both 1, hence $M(w,s)f_{s}^v=c(w,s)f_{ws}^v$ for some constant $c(w,s)$. Evaluating at $g=1$ on both sides, we get $M(w,s)f_{s}^v(1)=c(w,s)v$. For $w=w_i$, there is a map from $I(P,\sigma,s)$ to $I(\sigma,s_i)$ given by restricting functions on $G$ to $G_{\alpha_{i}}$, where $I(\sigma,s_i)$ is the space of functions $f:G_{\alpha_i}\to V$ such that $f(namx)=a^{s_i+1}\sigma(m)f(x)$(Here $a$ stands for $h_{\alpha_i}(a)$). Since $N_{w_i}=N\cap w_iNw_i^{-1}\backslash N=N_{\alpha_i}$, $M(w_i,s)$ induces a map from $I(\sigma,s_i)$ to $I(\sigma,-s_i)$. $f_{s}^v|_{G_{\alpha_{i}}}$ satisfies $f_{s}^v(nak)=a^{s_i+1}\mu_{\sigma}(k)v$, $n\in N_{\alpha_i},a\in A_{\alpha_i},k\in K_{\alpha_i}$.

When $\Phi$ is simply laced or of type $G_2$, all roots are metaplectic by proposition 2.12. By proposition 6.4,
$$\mu_{\sigma}|_{K_{\alpha_i}}=m\sigma_{1/2}\oplus m'\sigma_{-1/2}$$ for some positive integers $m,m'$. Applying proposition 5.7, we see that $c(w_{i},s)=c_{1/2}(s_i)$.

Now assume $\Phi$ is of type $B_n,C_n,F_4$. If $\alpha_i$ is long, then it is metaplectic, by the same argument as the paragraph above, we have $c(w_{i},s)=c_{1/2}(s_i)$; If $\alpha_i$ is short, then it is not metaplectic by proposition 2.12. Hence by proposition 6.4,
$$\mu_{\sigma}|_{K_{\alpha_i}}=m\cdot 1$$ for some positive integer $m$. Applying proposition 5.7 again, we get $c(w_{i},s)=c_0(s_i)$.
\end{proof}

\begin{remark}
For any $w\in W$, write it as a reduced product of simple reflections: $w=w_1w_2\cdots w_n$. Then by proposition 4.3,
$$M(w,s)=M(w_1,w_2\cdots w_ns)M(w_2,w_3\cdots w_ns)\cdots M(w_{n-1},w_ns)M(w_n,s)$$ which implies
$$c(w,s)=c(w_1,w_2\cdots w_ns)c(w_2,w_3\cdots w_ns)\cdots c(w_{n-1},w_ns)c(w_n,s)$$
Define $$\overline M(w,s)=\frac{M(w,s)}{c(w,s)}$$ then $$\overline M(ww',s)=\overline M(w,w's)\circ \overline M(w',s)$$ for any $w,w'\in W$. They are called \emph{normalized intertwining operators} whose composition law behaves like the Weyl group.
\end{remark}

\end{document}